\documentclass[10pt]{amsart}
\usepackage{amssymb}
\usepackage{amsmath,amssymb,amsfonts,amsthm,
latexsym, amscd, amsfonts, epsfig, eepic,epic}
\usepackage{mathrsfs}
\usepackage{color}
\usepackage{eucal}%    caligraphic-euler fonts: \mathcal{ }
\usepackage{eufrak}%   frak-euler        fonts: \mathfrak{ }
\usepackage[all]{xypic}
\usepackage{xspace}

\theoremstyle{plain}
\newtheorem{theorem}{Theorem}

\newtheorem{lemma}{Lemma}

\theoremstyle{definition}

\theoremstyle{example}

\theoremstyle{remark}

\numberwithin{equation}{section}

\begin{document}

%%%
%%%
%%%%%%%%%%%%%%%%%%%%%%%%%%%%%%%%%%%%%%%%%%%%%%%%%%%%%%%%%%%%%%%%%%%%%%%%%%
%%
%%%

\title[Canonical RNA pseudoknot structures]
{Canonical RNA pseudoknot structures}
\author{Gang Ma and Christian M. Reidys$^{\,\star}$}
\address{Center for Combinatorics, LPMC-TJKLC %XXX%
           \\
         Nankai University  \\
         Tianjin 300071\\
         P.R.~China\\
         Phone: *86-22-2350-6800\\
         Fax:   *86-22-2350-9272}
\email{reidys@nankai.edu.cn}%XXXX
\thanks{}
\date{June, 2008}
\begin{abstract}
In this paper we study $k$-noncrossing, canonical RNA pseudoknot
structures with minimum arc-length $\ge 4$. Let ${\sf
T}_{k,\sigma}^{[4]} (n)$ denote the number of these structures. We
derive exact enumeration results by computing the generating
function ${\bf T}_{k,\sigma}^{[4]}(z)= \sum_n{\sf
T}_{k,\sigma}^{[4]}(n)z^n$ and derive the asymptotic formulas
${\sf T}_{k,3}^{[4]}(n)^{}\sim c_k\,n^{-(k-1)^2-\frac{k-1}{2}}
(\gamma_{k,3}^{[4]})^{-n}$ for $k=3,\ldots,9$. In particular we
have for $k=3$, ${\sf T}_{3,3}^{[4]}(n)^{}\sim c_3\,n^{-5}
2.0348^n$. Our results prove that the set of biophysically
relevant RNA pseudoknot structures is surprisingly small and
suggest a new structure class as target for prediction algorithms.
\end{abstract}
\maketitle
{{\small
%\tableofcontents
}}

%%%
%%%
%%%%%%%%%%%%%%%%%%%%%%%%%%%%%%%%%%%%%%%%%%%%%%%%%%%%%%%%%%%%%%%%%%%%%%%%
%%%
%%%

\section{Introduction}\label{S:intro}

%%%
%%%%%%%%%%%%%%%%%%%%%%%%%%%%%%%%%%%%%%%%%%%%%%%%%%%%%%%%%%%%%%%%%%%%%%%%
%%%
RNA pseudoknot structures have drawn a lot of attention over the
last decade \cite{Science:05a}. From micro-RNA binding to
ribosomal frameshifts \cite{Parkin:91}, we currently discover
novel RNA functionalities at truly amazing rates. Our conceptional
understanding of RNA pseudoknot structures has not kept up with
this pace. Only recently the generating functions of
$k$-noncrossing RNA structures of arc-length $\ge 2$
\cite{Reidys:07pseu}, arc-length $\ge 4$ \cite{Han:08arc4} and
canonical $k$-noncrossing RNA structures
of arc-length $\ge 2$ \cite{Reidys:07lego} have been derived.
While these combinatorial results open new perspectives
for the design of new folding algorithms, it has to be noted that
realistic pseudoknot structures are subject to a minimum
arc-length $\ge 4$ and stack-length $\ge 3$. Therefore the above
structure classes are not ``best possible''. The lack of a transparent
target class of RNA pseudoknot structures represents a problem for
{\it ab initio} prediction algorithms.
There are four algorithms, capable of the energy based prediction of certain
pseudoknots in polynomial time: Rivas~{\it et al.} (dynamic
programming, gap-matrices, $O(n^6)$ time and $O(n^4)$ space)
\cite{Rivas:99a}, Uemura {\it et al.} ($O(n^5)$ time and $O(n^4)$
space, tree-adjoining grammars) \cite{Uemura:99a}, Akutsu
\cite{Akutsu:00a} and Lyngso \cite{Lyngso:00a}. All of them follow
the dynamic programming paradigm and none produces an easily
specifiable class of pseudoknots as output.

In this paper we
characterize a class of pseudoknot RNA structures in which bonds
have a minimum length of four and stacks contains at least three
base pairs. Our results show that this structure class is ideally
suited as {\it a priori}-output for prediction algorithms.
Tab.\ref{T:1} indicates that this class remains suitable even for
more complex pseudoknots (specified in terms of larger sets of
{\it mutually} crossing bonds). In fact, one can search RNA
$3$-noncrossing pseudoknot structure with arc-length $\ge 4$ and stack-length
$\sigma\ge 3$ for a sequence of length $100$ w.r.t.~a variety of
objective functions (in particular loop-based minimum free energy models)
on a $4$-core PC in a few minutes \cite{Reidys:08algo}.

In order to put our results into context, we turn the clock back by almost
three decades. 1978 M.~Waterman {\it et al.}
\cite{Waterman:79a,Waterman:78a,Waterman:80,Waterman:86} began deriving the
concepts for enumeration and prediction of RNA secondary structures.
The latter represent arguably {\it the} prototype of prediction-targets of
RNA structures.
RNA secondary structures are coarse grained structures which can be
represented as outer-planar graphs, diagrams, Motzkin-paths or words over
``$.$'' ``$\,(\,$'' and ``$\,)\,$''.
Their decisive feature is that they have no two crossing
bonds, see Fig.\ref{F:ma_reid1.eps}.
%%%
%%%%%%%%%%%%%%%%%%%%%%%%%%%%%%%%%%%%%%%%%%%%%%%%%%%%%%%%%%%%%%%%%%%%%%%%%%
%%%
\begin{figure}[ht]
\centerline{%
\epsfig{file=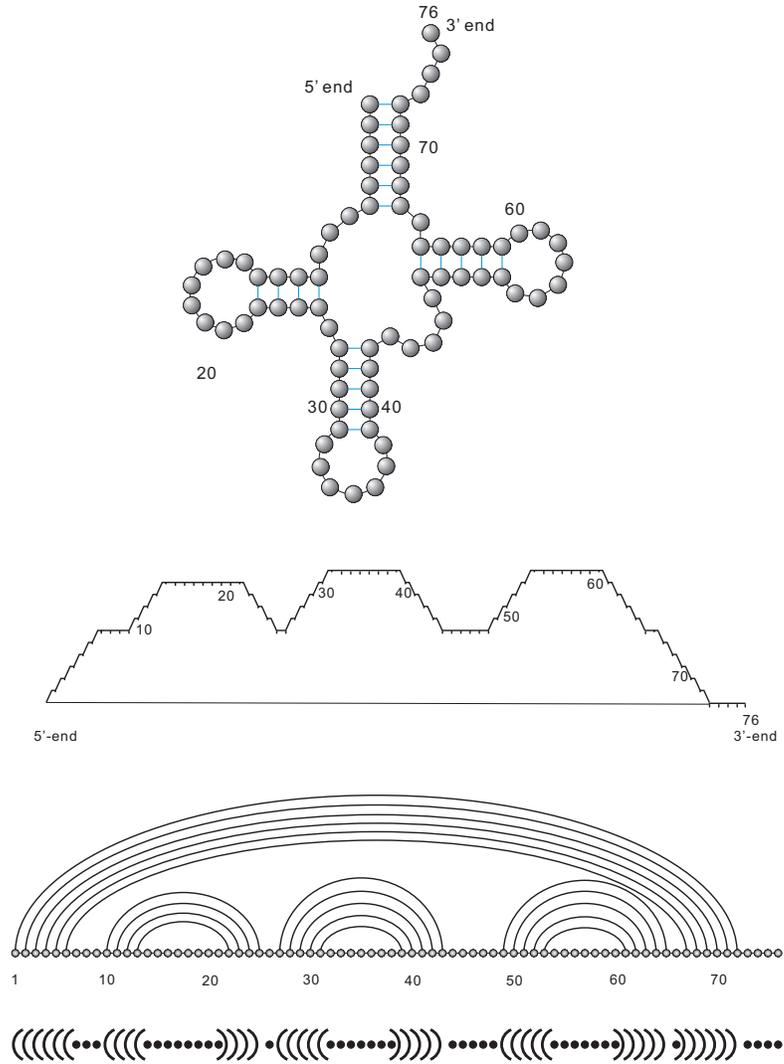,width=0.7\textwidth}\hskip15pt
 }
\caption{\small RNA secondary structures. } \label{F:ma_reid1.eps}
\end{figure}
%%%
%%%%%%%%%%%%%%%%%%%%%%%%%%%%%%%%%%%%%%%%%%%%%%%%%%%%%%%%%%%%%%%%%%%%%%%%%%
%%%
Let ${\sf T}_2^{[\lambda]}(n)$ denote the number of secondary
structures with arc-length $\ge \lambda$ over $[n]=\{1,\dots,n\}$.
The key to RNA secondary structures is the following recursion for
${\sf T}_2^{[\lambda]}(n)$:
\begin{equation}\label{E:sn}
{\sf T}_2^{[\lambda]}(n)={\sf T}_2^{[\lambda]}(n-1)+
\sum_{j=0}^{n-(\lambda+1)}{\sf T}_2^{[\lambda]}(n-2-j){\sf T}_2^{[\lambda]}(j),
\end{equation}
where ${\sf T}_2^{[\lambda]}(n)=1$ for $0\le n\le \lambda$. The latter follows from
considering the concatenation of Motzkin-paths with minimum peak length
$\lambda-1$.
Eq.~(\ref{E:sn})
implies for the generating function ${\bf T}_2^{[\lambda]}(z)=
\sum_{n\geq 0}{\sf T}_2^{[\lambda]}(n)z^n$ the functional equation
\begin{equation}\label{f}
z^2\, {\bf T}_2^{[\lambda]}(z)^2-
(1-z+z^2+\cdots+z^\lambda){\bf T}_2^{[\lambda]}(z)+1=0
\end{equation}
from which eventually
$$
{\bf T}_2^{[\lambda]}(z)
=\frac{-1+2z-2z^2+z^{\lambda+1}+\sqrt{1-4z+4z^2-2z^{\lambda+1}+
4z^{\lambda+2}-4z^{\lambda+3}+z^{2\lambda+2}}}
{2(z^3-z^2)}
$$
follows. Therefore, minimum arc-length restrictions do not impose
particular difficulties for RNA secondary structures. In fact
minimum stack size conditions can also be dealt with straightforwardly.
We furthermore note that eq.~(\ref{E:sn}) is a {\it constructive}
recursion, i.e.~it allows to inductively build secondary structures
over $[n]$ from those over $[i]$, for all $i<n$.

In order to analyze RNA structure with crossing
bonds, we recall the notion of $k$-noncrossing diagrams \cite{Reidys:07pseu}.
%%%
%%%%%%%%%%%%%%%%%%%%%%%%%%%%%%%%%%%%%%%%%%%%%%%%%%%%%%%%%%%%%%%%%%%%%%%%%%
%%%
\begin{figure}[ht]
\centerline{%
\epsfig{file=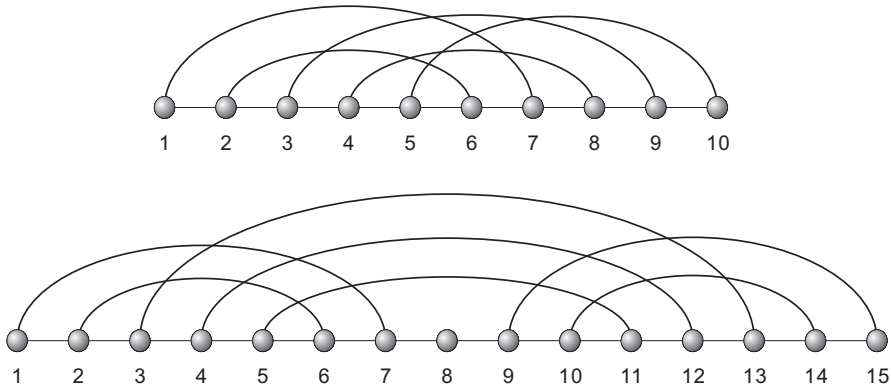,width=0.8\textwidth}\hskip15pt
 }
\caption{\small $k$-noncrossing diagrams: we display a $4$-noncrossing, arc-length
$\lambda\ge 4$ and $\sigma\ge 1$ (upper) and $3$-noncrossing, $\lambda\ge 4$
and $\sigma\ge 2$ (lower) diagram.} \label{F:ma_reid2.eps}
\end{figure}
%%%
%%%%%%%%%%%%%%%%%%%%%%%%%%%%%%%%%%%%%%%%%%%%%%%%%%%%%%%%%%%%%%%%%%%%%%%%%%
%%%
A $k$-noncrossing diagram is a labeled graph over the vertex set
$[n]$ with vertex degrees $\le 1$, represented by drawing its vertices
$1,\dots,n$ in a horizontal line and its arcs $(i,j)$, where $i<j$,
in the upper half-plane, containing at most $k-1$ mutually crossing
arcs.
The vertices and arcs correspond to
nucleotides and Watson-Crick ({\bf A-U}, {\bf G-C}) and ({\bf U-G})
base pairs, respectively. Diagrams have the following three key parameters:
the maximum number of mutually crossing arcs, $k-1$,
the minimum arc-length, $\lambda$ and minimum stack-length, $\sigma$
($\langle k,\lambda,\sigma\rangle$-diagrams).
The length of an arc $(i,j)$ is $j-i$ and a stack of length $\sigma$
is a sequence of ``parallel'' arcs of the form
$((i,j),(i+1,j-1),\dots,(i+(\sigma-1),j-(\sigma-1)))$, see
Fig.\ref{F:ma_reid2.eps}. We call an arc of length $\lambda$ a $\lambda$-arc.
Let ${T}_{k,\sigma}^{[\lambda]}(n)$ denote the set of $k$-noncrossing diagrams
with minimum arc- and stack-length $\lambda$ and $\sigma$
and let ${\sf T}_{k,\sigma}^{[\lambda]}(n)$ denote their number.

In the following, we shall identify pseudoknot RNA structures with
$k$-noncrossing diagrams and refer to them as
$\langle k,\lambda,\sigma\rangle$-structures.
Pseudoknot RNA structures occur in functional
RNA (RNAseP) \cite{Loria:96a}, ribosomal RNA
\cite{Konings:95a} and plant viral RNAs and vitro RNA evolution experiments
have produced families of RNA structures with pseudoknot motifs
\cite{Tuerk:92}.
In Fig.\ref{F:ma_reid3.eps} we give several representations of the
UTR-pseudoknot of the mouse hepatitis virus.
%%%
%%%%%%%%%%%%%%%%%%%%%%%%%%%%%%%%%%%%%%%%%%%%%%%%%%%%%%%%%%%%%%%%%%%%%%%%%%
%%%
\begin{figure}[ht]
\centerline{%
\epsfig{file=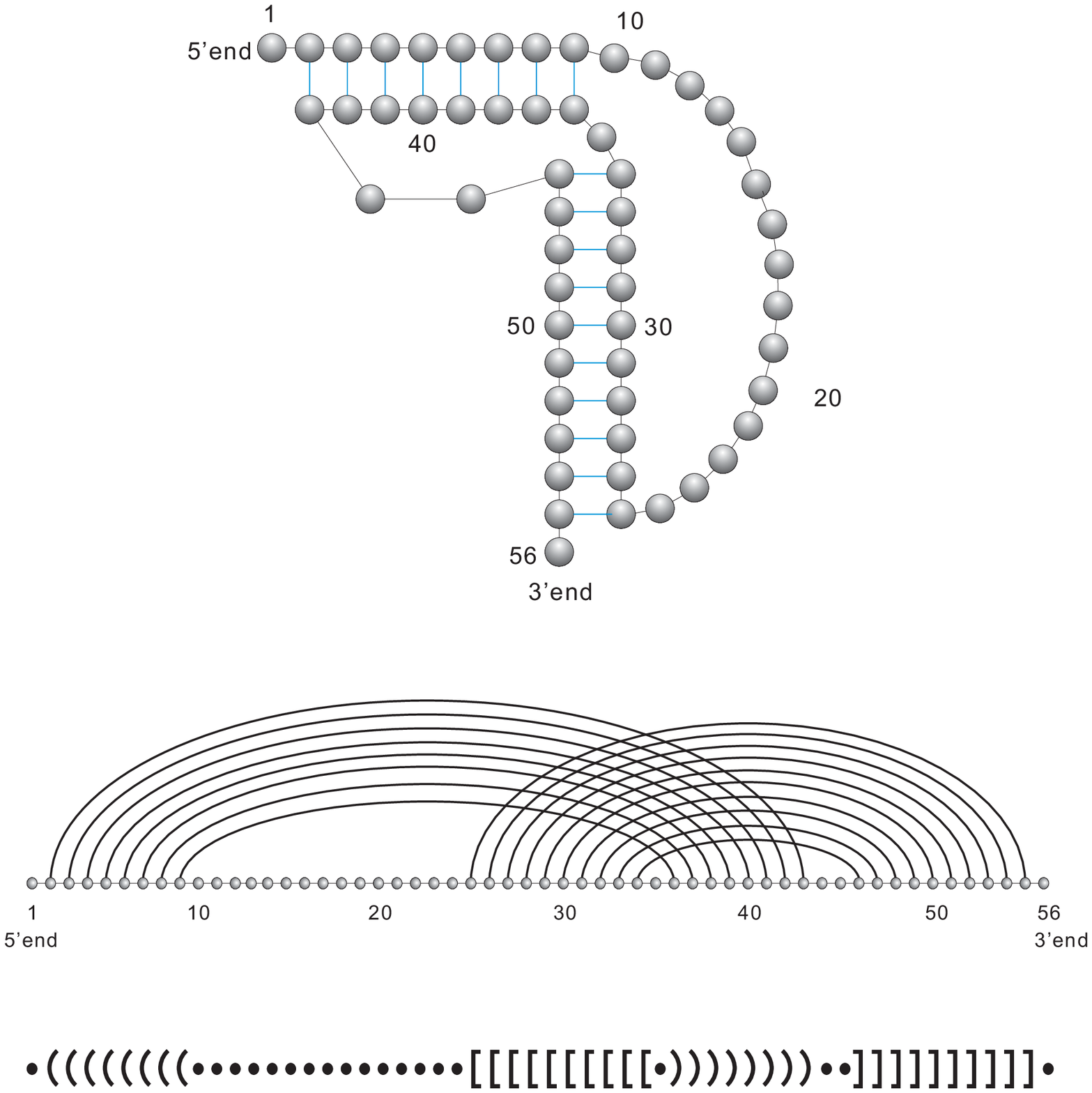,width=0.7\textwidth}\hskip15pt
 }
\caption{\small UTR-pseudoknot structure of the mouse hepatitis virus. }
\label{F:ma_reid3.eps}
\end{figure}
%%%
%%%%%%%%%%%%%%%%%%%%%%%%%%%%%%%%%%%%%%%%%%%%%%%%%%%%%%%%%%%%%%%%%%%%%%%%%%
%%%
Due to the crossings of arcs pseudoknots differs considerably from
secondary structures: pseudoknot RNA structures are inherently
non-inductive and no analogue of eq.~(\ref{E:sn}) exists.
One key for the generating function of $k$-noncrossing RNA structures
${\bf T}_k^{[\lambda]}(z)$ was the bijection of Chen {\it et al.}
\cite{Chen} obtained in the context of $k$-noncrossing partitions.
This bijection has been generalized to $k$-noncrossing tangled diagrams
\cite{Reidys:07vac}, a class of contact-structures tailored for expressing
RNA tertiary interactions.
Via the bijection $k$-noncrossing RNA structures can be identified with
certain walks in $\mathbb{Z}^{k-1}$ that remain in the region
$$
\{(x_1,\dots,x_{k-1})\in \mathbb{Z}^{k-1} \mid x_1\ge x_2\ge \dots x_{k-1}
\ge 0\}
$$
starting and ending at $0$, the boundaries of which
are called walls. The enumeration of these walks is obtained employing
the reflection principle. This method is due to Andr$\acute{e}$ in
$1887$ \cite{andre} and has subsequently been generalized by Gessel
and Zeilberger \cite{Zeilberger}.
In the reflection principle ``bad''-i.e.~reflected- walks cancel themselves.
In other words one enumerates all walks and due to cancellation only the ones
survive that never touch the walls. Despite its beauty this method does not
trigger any algorithmic intuition and is nonconstructive.
Moreover, $k$-noncrossing RNA structures cannot directly be enumerated via
the reflection principle: it does not preserve a minimum arc-length.
In \cite{Reidys:07pseu} it is shown how to eliminate specific classes of arcs
after reflection.
One nontrivial implication of this theory is that all generating
functions for $k$-noncrossing RNA structures are $D$-finite,
i.e.~there exists a {\it nonconstructive} recurrence relation of
finite length with polynomial coefficients for
${\sf T}_{k,\sigma}^{[\lambda]}(n)$.
Note however, that although we can prove the existence of this recurrence
it is at present not known for {\it any} $k>2$.
In Fig.\ref{F:ma_reid4.eps} we illustrate the key steps for
the enumeration of $k$-noncrossing RNA structures \cite{Reidys:07pseu}.
%%%
%%%%%%%%%%%%%%%%%%%%%%%%%%%%%%%%%%%%%%%%%%%%%%%%%%%%%%%%%%%%%%%%%%%%%%%%%%
%%%
\begin{figure}[ht]
\centerline{%
\epsfig{file=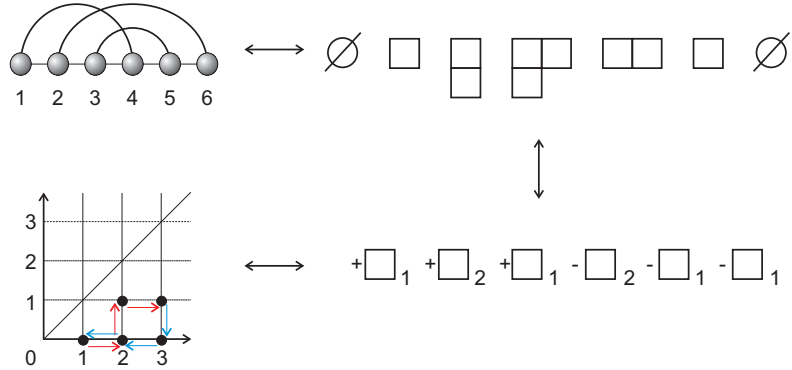,width=0.7\textwidth}\hskip15pt
 }
\caption{\small Exact enumeration of $k$-noncrossing RNA
structures.} \label{F:ma_reid4.eps}
\end{figure}
%%%
%%%%%%%%%%%%%%%%%%%%%%%%%%%%%%%%%%%%%%%%%%%%%%%%%%%%%%%%%%%%%%%%%%%%%%%%%%
%%%

Once ${\bf T}_{k,\sigma}^{[4]}(z)$ is known we employ singularity
analysis and study its dominant singularities, using Hankel
contours. This Ansatz has been pioneered by P.~Flajolet and A.M.~Odlyzko
\cite{Flajolet:07a}. Its basic idea is the construction of an
``singular-analogue'' of the Taylor-expansion. It can be shown
that, under certain conditions, there exists an approximation,
which is locally of the same order as the original function. The
particular, local approximation allows then to derive the
asymptotic form of the coefficients.
In our situation all conditions for singularity
analysis are met, since all our generating functions are $D$-finite
\cite{Stanley:80,Zeilberger:90} and $D$-finite functions have an
analytic continuation into any simply-connected domain containing
zero.

We will compute ${\bf T}_{k,\sigma}^{[4]}(z)$ and show that
${\bf T}_{k,\sigma}^{[4]}(z)$ has an unique dominant
singularity, whose type depends solely on the crossing number
\cite{Reidys:07asy1,Reidys:07lego}. Via singularity analysis will
produce an array of exponential growth rates indexed by $k$
and $\sigma$, summarized in Tab.~\ref{T:1}.
\begin{table}
\begin{center}
\begin{tabular}{|c|c|c|c|c|c|c|c|}
%\hline
%  \multicolumn{9}{|c|}{\textbf{$\lambda=2$}}\\
\hline $k$ & \small$3$ & \small $4$ & \small $5$ & \small $6$
&\small $7$ & \small $8$ & \small $9$  \\
\hline $\sigma=3$ & \small$2.0348$ & \small $2.2644$ & \small
$2.4432$ & \small $2.5932$
&\small $2.7243$ & \small $2.8414$ & \small $2.9480$  \\
$\sigma=4$ & \small$1.7898$ & \small $1.9370$ & \small $2.0488$ &
\small $2.1407$
&\small $2.2198 $ & \small $2.2896$ & \small $2.3523$  \\
$\sigma=5$ & \small$1.6465$ & \small $1.7532$ & \small $1.8330$ &
\small $1.8979$
&\small $1.9532$ & \small $2.0016$ & \small $2.0449$  \\
$\sigma=6$ & \small$1.5515$ & \small $1.6345$ & \small $1.6960$ &
\small $1.7457$
&\small $1.7877$ & \small $1.8243$ & \small $1.8569$  \\
 $\sigma=7$ & \small$1.4834$ & \small $1.5510$ & \small
$1.6008$ & \small $1.6408$
&\small $1.6745$ & \small $1.7038$ & \small $1.7297$  \\
$\sigma=8$ & \small$1.4319$ & \small $1.4888$ & \small $1.5305$ &
\small $1.5639$
&\small $1.5919$ & \small $1.6162$ & \small $1.6376$  \\
 $\sigma=9$ & \small$1.3915$ & \small $1.4405$ & \small
$1.4763$ & \small $1.5049$
&\small $1.5288$ & \small $1.5494$ & \small $1.5677$  \\
\hline
\end{tabular}
\centerline{}   \caption{\small Exponential growth rates of
$\langle k, 4, \sigma\rangle$-structures where $\sigma\ge 3$.
}\label{T:1}
\end{center}
\end{table}
The ideas of this paper build on those of
\cite{Reidys:07pseu,Reidys:07lego}. In \cite{Reidys:07lego}
core-structures are introduced via which $\sigma$-canonical
$k$-noncrossing structures can be enumerated. $\langle
k,4,\sigma\rangle$-structures where $\sigma\ge 3$ can however not
be enumerated via core-structures, see Fig.\ref{F:ma_reid5.eps}.
\begin{figure}[ht]
\centerline{\epsfig{file=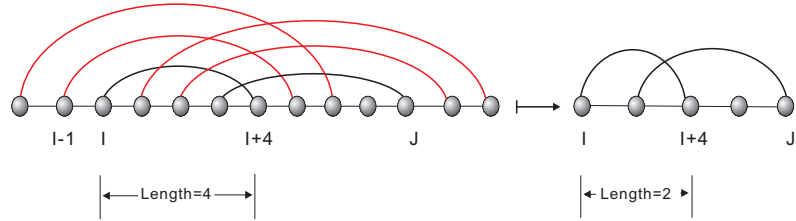,width=0.7\textwidth}\hskip15pt}
\caption{\small Core-structures will in general have $2$-arcs: the
structure $\delta\in T_{3,3}^{[4]}(12)$ (lhs) is mapped into its
core $c(\delta)$ (rhs). Clearly $\delta$ has arc-length $\ge 4$
and as a consequence of the collapse of the stack
$((I+1,J+2),(I+2,J+1),(I+3,J))$ (the red arcs are being removed)
into the arc $(I+3,J)$, $c(\delta)$ contains the arc $(I,I+4)$,
which is, after relabeling, a $2$-arc.}\label{F:ma_reid5.eps}
\end{figure}
This is a result from the fact that the core-map, obtained by
identifying stacks by single arcs does not preserve arc-length.
Therefore we have to introduce a new set of $k$-noncrossing
diagrams, denoted by $T_{k}^*(n,h)$. This class is designed for
inducing a new type of cores, $C_k^*(n',h')$ (see
Theorem~\ref{T:core-3}). Then we proceed using ideas similar to
those in \cite{Reidys:07lego} and prove our exact enumeration
result, Theorem~\ref{T:core-3}. As for the singularity analysis
the main contribution is Claim $1$ of Theorem~\ref{T:arc-3}: a new
functional equation for ${\bf T}_{k,\sigma}^{[4]}(z)$.

%%%
%%%%%%%%%%%%%%%%%%%%%%%%%%%%%%%%%%%%%%%%%%%%%%%%%%%%%%%%%%%%%%%%%%%%%%%%%%
%%%

\section{Preliminaries}\label{S:pre}

%%%
%%%%%%%%%%%%%%%%%%%%%%%%%%%%%%%%%%%%%%%%%%%%%%%%%%%%%%%%%%%%%%%%%%%%%%%%%%
%%%
In this Section we provide some background on the generating
functions of $k$-noncrossing matchings \cite{Chen,Wang:07} and
$k$-noncrossing RNA structures \cite{Reidys:07pseu,Reidys:07asy1}.
We denote the set (number) of $k$-noncrossing RNA structures with
arc-length $\ge\lambda$ and stack-size $\ge\sigma$ by
$T_{k,\sigma}^{[\lambda]}(n)$ (${\sf
T}_{k,\sigma}^{[\lambda]}(n)$). By abuse of notation we omit the
indices $\lambda$ and $\sigma$ in $T_{k,\sigma}^{[\lambda]}(n)$
(${\sf T}_{k,\sigma}^{[\lambda]}(n)$) for $\lambda=2$ and
$\sigma=1$. A $k$-noncrossing core-structure is a $k$-noncrossing
RNA structures in which there exists {\it no} two arcs of the form
$(i,j),(i+1,j-1)$.  The set (number) of $k$-noncrossing
core-structures and $k$-noncrossing core-structures with exactly
$h$ arcs is denoted by $C_k(n)$(${\sf C}_k(n)$) and $C_k(n,h)$
(${\sf C}_k(n,h)$), respectively. Furthermore we denote by
$f_{k}(n,\ell)$ the number of $k$-noncrossing diagrams with
arbitrary arc-length and $\ell$ isolated vertices over $n$
vertices and set ${\sf
M}_k(n)=\sum_{\ell=0}^{n}f_{k}(n,\ell)$. That is, ${\sf
M}_k(n)$ is the number of all $k$-noncrossing partial matchings.
In Fig.\ref{arc4stack3-4} we display the various types of diagrams
involved.
%%%
%%%%%%%%%%%%%%%%%%%%%%%%%%%%%%%%%%%%%%%%%%%%%%%%%%%%%%%%%%%%%%%%%%%%%%%%%%
%%%
\begin{figure}[ht]
\centerline{\epsfig{file=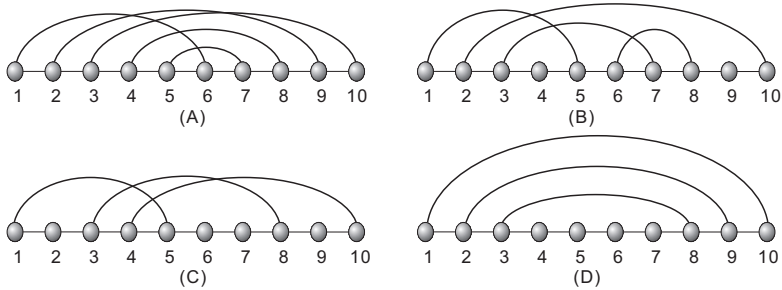,width=0.7\textwidth}\hskip8pt}
\caption{\small Basic diagram types: (A) $4$-noncrossing matching
(no isolated points), (B) $3$-noncrossing partial matching
(isolated points $4$ and $9$), (C) $4$-noncrossing RNA structure
with arc-length $\ge 4$ and stack length $\ge 1$, (D) RNA
structure with arc-length $\ge 5$ and stack-length $\ge
3$.}\label{arc4stack3-4}
\end{figure}
%%
%%%
%%%%%%%%%%%%%%%%%%%%%%%%%%%%%%%%%%%%%%%%%%%%%%%%%%%%%%%%%%%%%%%%%%%%%%%%%%
%%%
\subsection{$k$-noncrossing partial matchings and RNA structures}
The following identities are due to Grabiner and Magyar \cite{Grabiner:93a}
\begin{eqnarray}\label{E:ww0}
\label{E:ww1}
\sum_{n\ge 0} f_{k}(n,0)\cdot\frac{x^{n}}{n!} & = &
\det[I_{i-j}(2x)-I_{i+j}(2x)]|_{i,j=1}^{k-1} \\
\label{E:ww2}
\sum_{n\ge 0}\left\{\sum_{\ell=0}^nf_{k}(n,\ell)\right\}\cdot\frac{x^{n}}{n!}
&= &
e^{x}\det[I_{i-j}(2x)-I_{i+j}(2x)]|_{i,j=1}^{k-1} \ ,
\end{eqnarray}
where $I_{r}(2x)=\sum_{j \ge 0}\frac{x^{2j+r}}{{j!(r+j)!}}$
denotes the hyperbolic Bessel function of the first kind of order
$r$. Eq.~(\ref{E:ww1}) and (\ref{E:ww2}) allow only ``in
principle'' for explicit computation of the numbers $f_k(n,\ell)$
and in view of $f_{k}(n,\ell) ={n \choose \ell} f_{k}(n-\ell,0) $
everything can be reduced to (perfect) matchings, where we have
the following situation: there exists an asymptotic approximation
of the determinant of the hyperbolic Bessel function for general
$k$ due to \cite{Wang:07} and employing the subtraction of
singularities-principle \cite{Odlyzko:95a} one can prove
\cite{Wang:07}
\begin{equation}\label{E:f-k-imp}
\forall\, k\in\mathbb{N};\qquad  f_{k}(2n,0) \, \sim  \, c_k  \,
n^{-((k-1)^2+(k-1)/2)}\, (2(k-1))^{2n},\qquad c_k>0 \ ,
\end{equation}
where $\rho_k=\frac{1}{2(k-1)}$ is the dominant real singularity
of $\sum_{n\ge 0}f_k(2n,0)z^{2n}$. For $\langle
k,2,1\rangle$-structures we have
\cite{Reidys:07pseu,Reidys:07asy1}
\begin{eqnarray}
\label{E:sum} {\sf T}_{k}^{}(n)  & = &  \sum_{b=0}^{\lfloor
n/2\rfloor}(-1)^{b}{n-b \choose b}\,{\sf M}_k(n-2b)  \\
{\sf T}_{k}^{}(n) & \sim & c_k^{}  \,
n^{-((k-1)^2+(k-1)/2)}\, (\gamma_k^{})^{-n},\qquad c_k^{}>0
\ ,
\end{eqnarray}
where $\gamma_k^{}$ is the unique, minimal solution of
$\frac{z}{z^2-z+1}=\rho_k$, see Tab.~\ref{T:tab1b}.
%%%
%%%%%%%%%%%%%%%%%%%%%%%%%%%%%%%%%%%%%%%%%%%%%%%%%%%%%%%%%%%%%%%%%%%%%%%%%%
%%%
\begin{table}
\begin{center}
\begin{tabular}{|c|c|c|c|c|c|c|c|c|c|c|}
\hline
       $k$ & $2$ & $3$  & $4$ &$5$ &$6$ &$7$ & $8$ & $9$ & $10$   \\
\hline $\gamma^{-1}_k$ &2.6180& 4.7913&
6.8541 & 8.8875 & 10.9083 & 12.9226 & 14.9330 & 16.9410 & 18.9472 \\
\hline
\end{tabular}
\end{center}
\centerline{}  \caption{\small The exponential growth rates of
$\langle k,2,1\rangle$-structures.} \vspace*{-11pt}
\label{T:tab1b}
\end{table}
%%%
%%%%%%%%%%%%%%%%%%%%%%%%%%%%%%%%%%%%%%%%%%%%%%%%%%%%%%%%%%%%%%%%%%%%%%
%%%
For $\langle k,4,1\rangle$-structures we have according to
\cite{Han:08arc4} the following exact enumeration result
\begin{eqnarray}\label{E:KK}
{\sf T}_{k}^{[4]}(n) & = & \sum_{b \le \lfloor \frac{n}{2}\rfloor}
(-1)^b\, \lambda(n,b)\, {\sf M}_k(n-2b), \qquad 4\le k\le 9 \ ,
\end{eqnarray}
where $\lambda(n,b)$ denotes the number of way of selecting $b$
arcs of length $\le 3$ over $n$ vertices and
\begin{equation}
 {\sf T}_{k}^{[4]}(n) \sim c_k\, n^{-((k-1)^2+(k-1)/2)}\,
\left({\gamma_k^{[4]}}\right)^{-n} \,
\end{equation}
where $\gamma_k^{[4]}$ is the unique positive, real solution of
$\frac{z\, r_{1}(-z^2)}{1-zr_{1}(-z^2)}=\rho_k$ where $r_1(z)$ satisfies
\begin{eqnarray*}
u(z)   & = & \sqrt{1+4z-4z^2-6z^3+4z^4+z^6} \\
\label{E:oben} r_1(z) & = &
-\frac{-2z^2+z^3-1+u(z)}{2(1-2z-z^2+z^4)} \ .
\end{eqnarray*}
In Tab.~\ref{T:tab2} we present the exponential growth rates for
${\sf T}_k^{[4]}(n)$ for $k=4,\dots,9$.
%%%
%%%%%%%%%%%%%%%%%%%%%%%%%%%%%%%%%%%%%%%%%%%%%%%%%%%%%%%%%%%%%%%%%%%%%%%%%%
%%%
\begin{table}
\begin{center}
\begin{tabular}{|c|c|c|c|c|c|c|}
\hline
 $k$   & $4$ &$5$ &$6$ &$7$ &
$8$ & $9$   \\
\hline ${\gamma_k^{[4]}}^{-1}$& $6.5290$ & $8.6483$ & $10.7176$ &
$12.7635$ & $14.7963$ & $16.8210$  \\  \hline
\end{tabular}
\centerline{}  \caption{\small The exponential growth rates of
$\langle k,4,1\rangle$-structures.} \label{T:tab2}\vspace*{-12pt}
\end{center}
\end{table}
%%%
%%%%%%%%%%%%%%%%%%%%%%%%%%%%%%%%%%%%%%%%%%%%%%%%%%%%%%%%%%%%%%%%%%%%%%
%%%
For $\langle k,2,\sigma\rangle$-structures
we have according to
\cite{Reidys:07lego}
\begin{equation}
{\bf T}_{k,\sigma}(x)=\frac{1}{u_0x^2-x+1}\sum_{n\ge
0}f_k(2n,0)\left(\frac{\sqrt{u_0}\,x}{u_0x^2-x+1}\right)^{2n}
\end{equation}
where $u_0=\frac{(x^2)^{\sigma-1}}{(x^2)^\sigma-x^2+1}$ and
\begin{equation}
{\sf T}_{k,\sigma}(n)\sim c_k\, n^{-((k-1)^2+(k-1)/2)}\,
\left(\gamma_k^{-1}\right)^n \,
\end{equation}
where $\gamma_{k,\sigma}$ is a positive real dominant singularity
of $\sum_{n\ge 0}{\sf T}_{k,\sigma}(n)x^n$ and the minimal
positive real solution of the equation
\begin{equation}
\frac{\sqrt{\frac{(x^2)^{\sigma-1}}{(x^2)^{\sigma}-x^2+1}}\,x}
{\left(\frac{(x^2)^{\sigma-1}}{(x^2)^\sigma-x^2+1}\right)x^2-x+1}=\rho_k.
\end{equation}
%%%
%%%%%%%%%%%%%%%%%%%%%%%%%%%%%%%%%%%%%%%%%%%%%%%%%%%%%%%%%%%%%%%%%%%%%%%%%%%
%%%
In Table~\ref{T:tab3} we present the exponential growth rates of
$\langle k,2,\sigma\rangle$-structures.
\begin{table}
\begin{center}
{\small
\begin{tabular}{|l|c|c|c|c|c|c|c|c|c|}
%\hline
%  \multicolumn{9}{|c|}{\textbf{$\lambda=2$}}\\
\hline $\quad k$ &  $2$   & $3$   &  $4$   &$5$     &$6$     &  $7$    &  $8$    & $9$     & $10$\\
\hline
$\sigma=2$ & 1.9680 & 2.5881 & 3.0382 & 3.4138 & 3.7438 & 4.0420 & 4.3162 & 4.5715 & 4.8115  \\
$\sigma=3$ & 1.7160 & 2.0477 & 2.2704 & 2.4466 & 2.5955 & 2.7259 & 2.8427 & 2.9490 & 3.0469  \\
$\sigma=4$ & 1.5782 & 1.7984 & 1.9410 & 2.0511 & 2.1423 & 2.2209 & 2.2904 & 2.3529 & 2.4100  \\
$\sigma=5$ & 1.4899 & 1.6528 & 1.7561 & 1.8347 & 1.8991 & 1.9540 & 2.0022 & 2.0454 & 2.0845  \\
\hline
\end{tabular}
} \centerline{}   \caption{\small The exponential growth rates
$\langle k,2,\sigma\rangle$-structures
\cite{Reidys:07lego}.}\label{T:tab3}
\end{center}
\end{table}

\subsection{Singularity analysis}
Let us next recall some basic fact about analytic functions.
Pfringsheim's Theorem \cite{Titmarsh:39} guarantees that each
power series with positive coefficients has a positive real
dominant singularity. This singularity plays a key role for the
asymptotics of the coefficients. In the proof of
Theorem~\ref{T:arc-3} it will be important to deduce relations
between the coefficients from functional equations of generating
functions. The class of theorems that deal with such deductions
are called transfer-theorems \cite{Flajolet:07a}. We consider a
specific domain in which the functions in question are analytic
and which is ``slightly'' bigger than their respective radius of
convergence. It is tailored for extracting the coefficients via
Cauchy's integral formula. Details on the method can be found in
\cite{Flajolet:07a}. In case of $D$-finite functions we have
analytic continuation in any simply connected domain containing
zero \cite{Wasow:87} and all prerequisites of singularity analysis
are met. To be precise, given two numbers $\phi,R$, where $R>1$
and $0<\phi<\frac{\pi}{2}$ and $\rho\in\mathbb{R}$, the open
domain $\Delta_\rho(\phi,R)$ is defined as
\begin{equation}
\Delta_\rho(\phi,R)=\{ z\mid \vert z\vert < R, z\neq \rho,\,
\vert {\rm Arg}(z-\rho)\vert >\phi\}
\end{equation}
A domain is a $\Delta_\rho$-domain if it is of the form
$\Delta_\rho(\phi,R)$ for some $R$ and $\phi$.
A function is $\Delta_\rho$-analytic if it is analytic in some
$\Delta_\rho$-domain.
We use the notation
\begin{equation}\label{E:genau}
\left(f(z)=O\left(g(z)\right) \
\text{\rm as $z\rightarrow \rho$}\right)\quad \Longleftrightarrow \quad
\left(f(z)/g(z) \ \text{\rm is bounded as $z\rightarrow \rho$}\right)
\end{equation}
and if we write $f(z)=O(g(z))$ it is implicitly assumed that $z$
tends to a (unique) singularity. $[z^n]\,f(z)$ denotes the
coefficient of $z^n$ in the power series expansion of $f(z)$ around
$0$.

%%%
%%%%%%%%%%%%%%%%%%%%%%%%%%%%%%%%%%%%%%%%%%%%%%%%%%%%%%%%%%%%%%%%%%%%%%%%%
%%%
\begin{theorem}\label{T:transfer1}{\bf }
\cite{Flajolet:07a} Let $f(z),g(z)$ be $D$-finite,
$\Delta_{\rho}$-analytic functions with unique dominant singularity
$\rho$ and suppose
\begin{equation}
f(z) = O( g(z))\quad  \mbox{ for  }\quad z\rightarrow \rho \ .
\end{equation}
Then we have
\begin{equation}
[z^n]f(z)= \,K \,\left(1-O(\frac{1}{n})\right)\,  [z^n]g(z) \ ,
\end{equation}
where $K$ is some constant.
\end{theorem}
%%%
%%%%%%%%%%%%%%%%%%%%%%%%%%%%%%%%%%%%%%%%%%%%%%%%%%%%%%%%%%%%%%%%%%%%%%%%%
%%%
Let ${\bf F}_k(z)=\sum_{n}f_k(2n,0)z^{2n}$, the ordinary generating function
of $k$-noncrossing matchings. It follows from eq.~(\ref{E:ww1}) that the
power series ${\bf F}_k(z)$ is $D$-finite, i.e.~there exists some $e\in \mathbb{N}$
such that
\begin{equation}\label{E:JK0}
q_{0,k}(z)\frac
{d^e}{dz^e}{\bf F}_k(z)+q_{1,k}(z)\frac{d^{e-1}}{dz^{e-1}}{\bf F}_k(z)+\cdots+
q_{e,k}(z){\bf F}_k(z)=0  \  ,
\end{equation}
where $q_{j,k}(z)$ are polynomials. The key point is that any
dominant singularity of ${\bf F}_k(z)$ is contained in the set of
roots of $q_{0,k}(z)$ \cite{Stanley:80}, which we denote by $M_k$.
The polynomials $q_{0,k}(z)$ and their sets of roots for
$k=3,\dots,9$ are given in Table~\ref{T:111}. Accordingly, ${\bf
F}_k(z)$ has singularities $\pm \rho_k$, where
$\rho_k=(2(k-1))^{-1}$.
\begin{center}
\begin{table}
\begin{tabular}{|c|c|c|}
\hline
$k$ & $q_{0,k}(z)$ & $M_k$ \\
\hline $3$ & $({1}/{4}-4z^2)\, z^2$ & $\{\pm{1}/{4}\}$ \\
\hline $4$ & $(144\, z^4-40\, z^2 +1)  \, z^6$ &
$\{\pm{1}/{2},\pm{1}/{6}\}$\\
\hline $5$ & $(-80\, z^2+ 1024\, z^4 + 1)\, z^8$ &
$\{\pm{1}/{4},\pm{1}/{8}\}$\\
\hline $6$ & $(-4144\, z^4 + 140\, z^2+14400\, z^6 +1)\,z^{10}$ &
$\{\pm{1}/{2},\pm{1}/{6},\pm{1}/{10},
 \}$\\
\hline $7$ & $(-1-12544\, z^4+224\, z^2+147456\,z^6)\,z^{12}$ &
$\{\pm{1}/{4},\pm{1}/{8},\pm{1}/{12} \}$\\
\hline $8$ & $(1-336z^2+31584z^4+2822400z^8-826624z^6)z^{14}$ &
$\{\pm{1}/{2},\pm{1}/{6},\pm{1}/{10},\pm{1}/{14}\}$\\
\hline $9$ & $-(-480z^2+1+69888z^4+37748736z^8-3358720z^6)z^{16}$
&
$\{\pm{1}/{4},\pm{1}/{8},\pm{1}/{12},\pm{1}/{16}\}$\\
\hline
\end{tabular}
\centerline{}
\caption{\small The polynomials $q_{0,k}(z)$ and their nonzero roots.}
\label{T:111}
\end{table}
\end{center}
As a consequence of Theorem~\ref{T:transfer1},
eq.~(\ref{E:f-k-imp}) and the so called supercritical case of
singularity analysis \cite{Flajolet:07a}, VI.9., p.~400, we give
the following result\cite{Reidys:08lego2} tailored for our
functional equations.
%%%
%%%%%%%%%%%%%%%%%%%%%%%%%%%%%%%%%%%%%%%%%%%%%%%%%%%%%%%%%%%%%%%%%%%%%%%%%
%%%
\begin{theorem}\label{T:realdeal}
Suppose $\vartheta_{\sigma}(z)$ is algebraic over $K(z)$, regular
for $\vert z\vert <\delta$ and satisfies $\vartheta_{\sigma}(0)=0$.
Suppose further $\gamma_{k,\sigma}$ is the unique solution with
minimal modulus $<\delta$ of the two equations
$\vartheta_{\sigma}(x)=\rho_k$ and $\vartheta_{\sigma}(x)=-\rho_k$.
Then $\gamma_{k,\sigma}$ is the unique dominant singularity of
${\bf F}_k(\vartheta_{\sigma}(z))$ and
\begin{equation}
[z^n]\,{\bf F}_k(\vartheta_{\sigma}(z)) \sim c_k  \,
n^{-((k-1)^2+(k-1)/2)}\, \left(\gamma_{k,\sigma}^{-1}\right)^n \ .
\end{equation}
\end{theorem}
%%%
%%%%%%%%%%%%%%%%%%%%%%%%%%%%%%%%%%%%%%%%%%%%%%%%%%%%%%%%%%%%%%%%%%%%%%%%%
%%%

%%%
%%%%%%%%%%%%%%%%%%%%%%%%%%%%%%%%%%%%%%%%%%%%%%%%%%%%%%%%%%%%%%%%%%%%%%%%%
%%%

\section{Exact Enumeration}\label{S:core-3}

%%%
%%%%%%%%%%%%%%%%%%%%%%%%%%%%%%%%%%%%%%%%%%%%%%%%%%%%%%%%%%%%%%%%%%%%%%%%%
%%%

In this section we present the exact enumeration of
$\langle k,4,\sigma\rangle$-structures, where $\sigma\ge 3$.
The structure of
our formula is analogous to the M\"obius inversion formula proved
in \cite{Reidys:07lego}: $ {\sf T}_{k,\sigma}^{}(n,h)=
\sum_{b=\sigma-1}^{h-1} {b+(2-\sigma)(h-b)-1 \choose h-b-1} {\sf
C}_k(n-2b,h-b)$, which relates the number of all structures and
the number of core-structures. As we pointed out in the
introduction the latter cannot be used in order to enumerate
$k$-noncrossing structures with arc-length $\ge 4$,  see
Fig.\ref{F:ma_reid5.eps}. We consider the arc-sets
$$
\beta_2=\{(i,i+2)\mid i+1\ \text{\rm isolated}\}\quad\text{\rm and}\quad
\beta_3=
\{(i,i+3)\mid i+1,i+2\ \text{\rm isolated}\}
$$
and set $\beta=\beta_2\cup\beta_3$. Furthermore
\begin{eqnarray}\label{E:*}
C^*_k(n,h) & = & \{\delta\mid \delta\in C_k(n,h);\,
\text{\rm $\delta$ contains no $1$-arc and no $\beta$-arc}\}\\
T^*_k(n,h) & = & \{\delta\mid \delta\in T_k(n,h);\,
\text{\rm $\delta$ contains no $1$-arc and no $\beta$-arc}\} \ .
\end{eqnarray}
%%%
%%%%%%%%%%%%%%%%%%%%%%%%%%%%%%%%%%%%%%%%%%%%%%%%%%%%%%%%%%%%%%%%%%%%%%
%%%
\begin{theorem}\label{T:core-3}
Suppose we have $k,h,\sigma\in \mathbb{N}$, $k\ge 2$, $h\le n/2$
and $\sigma\ge 3$. Then the number of $\langle k,4,\sigma\rangle$-structures
having exactly $h$ arcs is given by
\begin{equation}\label{E:relation1}
{\sf T}_{k,\sigma}^{[4]}(n,h)=\sum_{b=\sigma-1}^{h-1}
{b+(2-\sigma)(h-b)-1 \choose h-b-1} {\sf C}_k^*(n-2b,h-b)
\end{equation}
where ${\sf C}_{k}^{*}(n,h)$ satisfies ${\sf C}_{k}^{*}(n,0)=1$ and
\begin{eqnarray}\label{E:relation2}
{\sf C}_{k}^{*}(n,h)=\sum_{b=0}^{h-1}(-1)^{h-b-1}{h-1\choose b}{\sf
T}_k^{*}(n-2h+2b+2,b+1) \mbox{\quad for } \ h\ge 1 \ .
\end{eqnarray}
Furthermore, ${\sf T}_k^{*}(n,h)$ satisfies
\begin{equation}\label{E:relation3}
{\sf T}_k^{*}(n,h)=\sum_{0\le j_1+j_2+j_3\le
h}(-1)^{j_1+j_2+j_3}\lambda(n,j_1,j_2,j_3)\,
f_{k}(n-2j_1-3j_2-4j_3,n-2h-j_2-2j_3)
\end{equation}
where
$$
\lambda(n,j_1,j_2,j_3)={n-j_1-2j_2-3j_3 \choose
j_1,j_2,j_3,n-2j_1-3j_2-4j_3} \ .
$$
\end{theorem}
%%%
%%%%%%%%%%%%%%%%%%%%%%%%%%%%%%%%%%%%%%%%%%%%%%%%%%%%%%%%%%%%%%%%%%%%%%
%%%
In Tab.\ref{T:000} we display the first numbers of
$\langle k,4,3\rangle$-structures and
$\langle k,4,4\rangle$-structures, respectively.
\begin{table}
\begin{center}
\begin{tabular}{|c|c|c|c|c|c|c|c|c|c|c|c|c|c|c|c|c|c|}
%\hline
%  \multicolumn{9}{|c|}{\textbf{$\lambda=2$}}\\
\hline $n$ & \small$8$ & \small $9$ & \small $10$ & \small $11$
&\small $12$ & \small $13$ & \small $14$ & \small $15 $ &
\small$16$ & \small$17$ & \small$18$ & \small$19$ & \small$20$ &
\small$21$
& \small$22$ & \small $23$ & \small $24$ \\
\hline \small ${\sf T}_{3,3}^{[4]}(n)$ & \small $1$ &\small $2$ &
\small$4$ & \small $8$ & \small $15$ & \small$28$ & \small$52$&
\small $96$ & \small$176$ & \small$316$ & \small$557$ &
\small$965$ &\small $1660$ & \small$2860$ & \small$4974$
& \small $8754$ & \small $15562$ \\
\hline \small ${\sf T}_{3,4}^{[4]}(n)$ & \small $1$ &\small $1$ &
\small$1$ & \small $2$ & \small $4$ & \small$8$ & \small$14$&
\small $23$ & \small$36$ & \small$56$ & \small$88$ & \small$141$
&\small $231$ & \small$382$ & \small$633$
& \small $1038$ & \small $1679$ \\
\hline
\end{tabular}
\centerline{} \caption{\small Exact enumeration: ${\sf
T}_{3,3}^{[4]}(n)$ and ${\sf T}_{3,4}^{[4]}(n)$ for $n\le 24$,
respectively.}\label{T:000}
\end{center}
\end{table}
\begin{proof}
We first show that there exists a mapping from $\langle
k,4,\sigma\rangle$-structures with $h$ arcs over $[n]$ into
$\dot\bigcup_{\sigma-1\le b\le h-1}C^*_k (n-2b,h-b)$:
\begin{equation}
c\colon T^{[4]}_{k,\sigma}(n,h)\rightarrow
\dot\bigcup_{\sigma-1\le b\le h-1}{C}^*_k(n-2b,h-b), \quad
\delta\mapsto c(\delta)
\end{equation}
which is obtained in two steps: first induce $c(\delta)$ by
mapping arcs and isolated vertices as follows:
\begin{equation}
\forall \,\ell\ge \sigma-1;\quad
((i-\ell,j+\ell),\dots,(i,j)) \mapsto (i,j) \ \quad \text{\rm and} \quad
j \mapsto j \quad \text{\rm if $j$ is an isolated vertex}
\end{equation}
and second relabel the resulting diagram from left to right in increasing
order, see Fig.\ref{arc4stack3-2}.\\
%%%
%%%%%%%%%%%%%%%%%%%%%%%%%%%%%%%%%%%%%%%%%%%%%%%%%%%%%%%%%%%%%%%%%%%%%%%%%%
%%%
\begin{figure}[ht]
\centerline{\epsfig{file=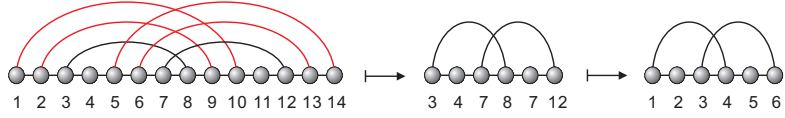,width=0.7\textwidth}\hskip8pt}
\caption{\small The mapping $c\colon
T_{k,\sigma}^{[4]}(n,h)\longrightarrow \dot\bigcup_{\sigma-1\le
b\le h-1}{C}_k^{*}(n-2b,h-b)$ is obtained in two steps: first
contraction of the stacks while keeping isolated points and
secondly relabeling of the resulting diagram.  }
\label{arc4stack3-2}
\end{figure}
%%%
%%%%%%%%%%%%%%%%%%%%%%%%%%%%%%%%%%%%%%%%%%%%%%%%%%%%%%%%%%%%%%%%%%%%%%%%%%
%%%
{\it Claim $1$.} $c\colon T^{[4]}_{k,\sigma}(n,h)\longrightarrow
\dot\bigcup_{\sigma-1\le b\le h-1}{C}^*_k(n-2b,h-b)$ is
well-defined and surjective.\\
By construction, $c$ does not change the crossing number. Since
$T^{[4]}_{k,\sigma}(n)$ contains only arcs of length $\ge 4$ we
derive $c(T_{k,\sigma}^{[4]}(n))\subset C^*_k(n-2b,h-b)$.
Therefore $c$ is well-defined. It remains to show that $c$ is
surjective. For this purpose let $\delta\in C^*_k(n-2b,h-b)$ and
set $a=b-(\sigma-1)(h-b)$. We proceed constructing a
$k$-noncrossing structure $\tilde{\delta}$ in
three steps: \\
{\it Step $1$.} replace each label $i$ by $r_i$, where $r_i\le r_s$ if
and only if $i\le s$.\\
{\it Step $2$.} replace the leftmost arc $(r_p,r_q)$ by the sequence of arcs
\begin{equation}
\left((\tau_p-([\sigma-1]+a),\tau_q+([\sigma-1]+a)),\dots,(\tau_p,\tau_q)
\right)
\end{equation}
replace any other arc $(r_p,r_q)$ by the sequence
\begin{equation}
\left((\tau_p-[\sigma-1],\tau_q+[\sigma-1]),\dots,(\tau_p,\tau_q)\right)
\end{equation}
and each isolated vertex $r_s$ by $\tau_s$.\\
{\it Step $3$.} Set for $x,y\in\mathbb{Z}$, $\tau_b+y \le \tau_c+ x$ if and
only if ($b<c$) or ($b=c$ and $y\le x$). By construction, $\le$ is a linear
order over
$$
n-2b+2(h-b)\,(\sigma-1)+2a=
n-2b+2(h-b)\,(\sigma-1)+2(b-(\sigma-1)(h-b))=n
$$
elements, which we then label from $1$ to $n$ (left to right) in
increasing order. It is straightforward to verify that
$c(\tilde{\delta})=\delta$ holds. It remains to show that
$\tilde{\delta}\in T^{[4]}_{k,\sigma}(n)$. Suppose {\it a
contrario} $\tilde{\delta}$ contains an arc $(i,i+2)$. Since
$\sigma\ge 3$ we can then conclude that $i+1$ is necessarily
isolated. The arc $(i,i+2)$ is mapped by $c$ into $(j,j+2)$ with
isolated point $j+1$, which is impossible by definition of
$C^*_k(n',h')$. It follows similarly that an arc of the form $(i,i+3)$
cannot be contained in $\tilde{\delta}$ and Claim $1$ follows.\\
Labeling the $h$ arcs of $\delta\in T^{[4]}_{k,\sigma}(n,h)$ from
left to right and keeping track of multiplicities gives rise to
the map
\begin{equation}\label{E:core}
f_{k,\sigma}^{} \colon T^{[4]}_{k,\sigma}(n,h) \rightarrow
\dot\bigcup_{\sigma-1\le b\le h-1} \left[C_k^{*}(n-2b,h-b) \times
\left\{(a_{j})_{1\le j\le h-b}\mid\sum_{j=1}^{h-b}a_{j}=b, \
a_{j}\ge \sigma-1 \right\}\right],
\end{equation}
given by $f_{k,\sigma}^{ }(\delta)=(c(\delta),(a_{j})_{1\le j\le
h-b})$. We can conclude that $f_{k,\sigma}^{ }$ is well-defined and
a bijection. We proceed computing the multiplicities of the
resulting core-structures \cite{Reidys:07lego}:
\begin{equation}\label{E:core-1}
\vert\{(a_j)_{1\le j\le b}\mid \sum_{j=1}^{h-b}a_j=b; \ a_j\ge
\sigma -1\}\vert = {b+(2-\sigma)(h-b)-1 \choose h-b-1} \ .
\end{equation}
Eq.~(\ref{E:core-1}) and eq.~(\ref{E:core}) imply
\begin{equation}
{\sf T}^{[4]}_{k,\sigma}(n,h)=\sum_{b=\sigma-1}^{h-1}
{b+(2-\sigma)(h-b)-1 \choose h-b-1}{\sf C}^{*}_{k}(n-2b,h-b) \ ,
\end{equation}
whence eq.~(\ref{E:relation1}).
Next we consider the map
\begin{equation}
c^{*}\colon T^{*}_k(n,h)\rightarrow \dot\bigcup_{0\le b\le
h-1}{C}^*_k(n-2b,h-b), \quad \delta\mapsto c^*(\delta)
\end{equation}
Indeed, $c^*$ is well defined, since any diagram in $T^{*}_k(n,h)$
can be mapped into a core structure without $1$- and $\beta$-arcs,
i.e.~into an element of $C^{*}_{k}(n',h')$. That gives rise to
\begin{eqnarray}\label{relation_2}
{\sf T}_k^{*}(n,h)=\sum_{b=0}^{h-1}{h-1 \choose b}{\sf
C}_{k}^{*}(n-2b,h-b)
\end{eqnarray}
and via M\"obius-inversion formula we obtain
eq.~(\ref{E:relation2}). It is straightforward to show there are
$\lambda(n,j_1,j_2,j_3)={n-j_1-2j_2-3j_3\choose
j_1,j_2,j_3,n-2j_1-3j_2-4j_3}$ ways to select $j_1$ 1-arcs ,$j_2$
$\beta_2$-arcs and $j_3$ $\beta_3$-arcs over $[n]$. Since removing
$j_1$ 1-arcs, $j_2$ $\beta_2$-arcs and $j_3$ $\beta_3$-arcs
removes $2j_1+3j_2+4j_3$ vertices, the number of configurations of
at least $j_1$ 1-arcs, $j_2$ $\beta_2$-arcs and $j_3$
$\beta_3$-arcs is given by
$\lambda(n,j_1,j_2,j_3)f_{k}(n-2j_1-3j_2-4j_3,n-2h-j_2-2j_3)$. Via
inclusion-exclusion principle, we arrive at
\begin{eqnarray*}
{\sf T}_k^{*}(n,h)=\sum_{0\le j_1+j_2+j_3\le
h}(-1)^{j_1+j_2+j_3}\lambda(n,j_1,j_2,j_3)f_{k}(n-2j_1-3j_2-4j_3,n-2h-j_2-2j_3)
\ ,
\end{eqnarray*}
whence Theorem~\ref{T:core-3}.
\end{proof}

The following functional identity,
relating the bivariate generating functions of ${\sf T}^{[4]}_{k,\sigma}(n,h)$
and ${\sf C}^{*}_{k}(n,h)$, is instrumental for proving our main result in the
next section, Theorem~\ref{T:arc-3}.

%%%
%%%%%%%%%%%%%%%%%%%%%%%%%%%%%%%%%%%%%%%%%%%%%%%%%%%%%%%%%%%%%%%%%%%%%%
%%%
\begin{lemma}\cite{Reidys:07lego}{\bf }\label{L:www-3}
Let $k,\sigma\in \mathbb{N}$,
$k\ge 2$ and let $u,x$ be indeterminants.
Suppose we have
\begin{equation}\label{E:uni1}
\forall\, h\ge 1,\ \ {\sf
A}_{k,\sigma}^{}(n,h)=\sum_{b=\sigma-1}^{h-1} {b+(2-\sigma)(h-b)-1
\choose h-b-1} {\sf B}_k(n-2b,h-b)\ \text{\it and} \  {\sf
A}_{k,\sigma}^{}(n,0)=1 \ .
\end{equation}
Then we have the functional relation
\begin{eqnarray}\label{E:universal}
\sum_{n \ge 0}\sum_{0\le h\le \frac{n}{2}}{\sf A}^{ }_{k,
\sigma}(n,h)u^hx^n & = & \sum_{n \ge 0}\sum_{0\le h \le
\frac{n}{2}}{\sf B}^{ }_{k} (n,h)\left(\frac{u\cdot (ux^2)^{
\sigma-1}}{1-ux^2}\right)^hx^n\ .
\end{eqnarray}
\end{lemma}
%%%
%%%%%%%%%%%%%%%%%%%%%%%%%%%%%%%%%%%%%%%%%%%%%%%%%%%%%%%%%%%%%%%%%%%%%%
%%%

According to Lemma \ref{L:www-3} eq.~{(\ref{relation_2})} and
eq.~(\ref{E:relation1}) we obtain the two functional identities
\begin{eqnarray}
\label{E:c} &&\sum_{n \ge 0}\sum_{0\le h\le \frac{n}{2}}{\sf
T}^{*}_{k}(n,h)u^hx^n  =  \sum_{n \ge 0}\sum_{0\le h \le
\frac{n}{2}}{\sf C}^{
*}_{k}(n,h)\left(\frac{u}{1-ux^2}\right)^hx^n\\
\label{E:d} &&\sum_{n \ge 0} {\sf T}^{[4] }_{k,\sigma}(n)x^n  =
\sum_{n \ge 0}\sum_{0\le h \le \frac{n}{2}}{\sf
C}^{*}_{k}(n,h)\left(\frac{ (x^2)^{\sigma-1}}{1-x^2}\right)^hx^n
 \mbox{\quad for } \sigma\ge 3\ .
\end{eqnarray}

%%%
%%%%%%%%%%%%%%%%%%%%%%%%%%%%%%%%%%%%%%%%%%%%%%%%%%%%%%%%%%%%%%%%%%%%%%%%%
%%%

\section{Asymptotic Enumeration}\label{S:arc-3}

%%%
%%%%%%%%%%%%%%%%%%%%%%%%%%%%%%%%%%%%%%%%%%%%%%%%%%%%%%%%%%%%%%%%%%%%%%%%%
%%%

In this section we study the asymptotics of
$\langle k,4,\sigma\rangle$-structures,
where $\sigma\ge 3$. We are particularly
interested in deriving simple formulas that can be used for
assessing the complexity of prediction algorithms for
$k$-noncrossing RNA structures. In order to state
Theorem~\ref{T:arc-3} below we introduce
\begin{eqnarray}
\label{E:w} w_0(x)       & = &
\frac{x^{2\sigma-2}}{1-x^2+x^{2\sigma}}\\
v(x)         & = & 1-x+w(x)x^2+w(x)x^3+w(x)x^4 \\
\label{E:v} v_0(x) & = & 1-x+w_0(x)x^2+w_0(x)x^3+w_0(x)x^4  .
\end{eqnarray}

%%%
%%%%%%%%%%%%%%%%%%%%%%%%%%%%%%%%%%%%%%%%%%%%%%%%%%%%%%%%%%%%%%%%%%%%%%%%%%%%%
%%%
\begin{theorem}\label{T:arc-3}
Let $k,\sigma\in\mathbb{N}$, $k,\sigma\ge 3$, $x$ be an
indeterminate and $\rho_k$ the dominant, positive real singularity
of $\sum_{n\ge 0}f_k(2n,0)z^{2n}$. Then ${\bf T}^{[4]}_{
k,\sigma}(x)$, the generating function of
$\langle k,4,\sigma\rangle$-structures is given by
\begin{equation}\label{E:functional-arc-3}
{\bf T}_{k,\sigma}^{[4]}(x) =
\frac{1}{v_0(x)} \sum_{n \ge
0}f_k(2n,0)\left(\frac{\sqrt{w_{0}(x)}\,x}{v_0(x)}\right)^{2n} \ .
\end{equation}
Furthermore
\begin{equation}\label{E:growth-3}
{\sf T}_{k,\sigma}^{[4]}(n) \sim c_k\,
n^{-(k-1)^2-\frac{k-1}{2}}\,
\left(\frac{1}{\gamma_{k,\sigma}^{[4]}}\right)^n\ ,\quad \text{\it
for}\quad k=3,4,\dots,9
\end{equation}
holds, where $\gamma_{k,\sigma}^{[4]}$ is the positive real
dominant singularity of ${\bf T}_{k,\sigma}^{[4]}(x)$ and the
minimal positive real solution of the equation
$\frac{\sqrt{w_{0}(x)}\,x}{v_0(x)} = \rho_k$ and $f_k(2n,0)\sim
n^{-(k-1)^2-\frac{k-1}{2}}\left(\frac{1} {\rho_k}\right)^{2n}$
(eq.~(\ref{E:f-k-imp})).
\end{theorem}
%%%
%%%%%%%%%%%%%%%%%%%%%%%%%%%%%%%%%%%%%%%%%%%%%%%%%%%%%%%%%%%%%%%%%%%%%%%%%%%%
%%%
\begin{proof}
In the following we will use the notation $w_{0}$ instead of
$w_0(x)$, eq.~(\ref{E:w}). The first step derives a functional
equation relating the bivariate generating functions of
$T^*_k(n,h)$ and $f_k(2h',0)$. For this purpose we use
eq.~(\ref{E:relation3}).\\
{\it Claim $1$.}
\begin{equation}\label{E:dagger}
\begin{split}
\sum_{n\ge 0}\sum_{h \le \frac{n}{2}}{\sf T}_k^{*}(n,h)w^hx^n
=\frac{1}{v(x)}\sum_{n \ge
0}f_k(2n,0)\left(\frac{\sqrt{w}x}{v(x)}\right)^{2n} \ .
\end{split}
\end{equation}
Set $\varphi_{m}(w)=\sum_{h\le \frac{m}{2}}{m \choose
2h}f_k(2h,0)w^h$. In order to prove Claim $1$ we compute
\begin{eqnarray*}
 && \sum_{n\ge 0}\sum_{h \le \frac{n}{2}}{\sf T}_k^{*}(n,h)w^hx^n\\
&=& \sum_{n \ge 0}\sum_{h \le \frac{n}{2}}\sum_{0\le
j_1+j_2+j_3\le
h}(-1)^{j_1+j_2+j_3}
\lambda(n,j_1,j_2,j_3)f_k(n-2j_1-3j_2-4j_3,n-2h-j_2-2j_3)w^hx^n \\
&=&\sum_{n \ge 0}\sum_{j_1+j_2+j_3\le\frac{n}{2}}
(-1)^{j_1+j_2+j_3}\lambda(n,j_1,j_2,j_3)x^n\\
&&\hspace{3cm}\times\sum_{h\ge j_1+j_2+j_3}f_k(n-2j_1-3j_2-4j_3,n-2h-j_2-2j_3)w^h\\
&=&\sum_{n \ge 0}\sum_{j_1+j_2+j_3 \le
\frac{n}{2}}(-1)^{j_1+j_2+j_3}\lambda(n,j_1,j_2,j_3)x^n\\
&&\hspace{3cm}\times\sum_{h\ge j_1+j_2+j_3}{n-2j_1-3j_2-4j_3
\choose
n-2h-j_2-2j_3}f_k(2(h-j_1-j_2-j_3),0)w^h\\
&=&\sum_{n \ge 0}\sum_{j_1+j_2+j_3 \le
\frac{n}{2}}(-1)^{j_1+j_2+j_3}\lambda(n,j_1,j_2,j_3)
w^{j_1+j_2+j_3}\varphi_{n-2j_1-3j_2-4j_3}(w)x^n \ .
\end{eqnarray*}
We interchange the summation over $j_1+j_2+j_3$ and $n$ and arrive
at
\begin{align*}
& \quad\sum_{j_1+j_2+j_3\ge 0}\sum_{n\ge
2j_1+3j_2+4j_3}(-1)^{j_1+j_2+j_3}{n-j_1-2j_2-3j_3
\choose j_1,j_2,j_3,n-2j_1-3j_2-4j_3}w^{j_1+j_2+j_3}
\varphi_{n-2j_1-3j_2-4j_3}(w)x^n\\
&=\sum_{j_1+j_2+j_3\ge
0}\frac{(-w)^{j_1+j_2+j_3}}{j_1!j_2!j_3!}\sum_{n \ge
2j_1+3j_2+4j_3}\frac{(n-j_1-2j_2-3j_3)!}{(n-2j_1-3j_2-4j_3)!}
\varphi_{n-2j_1-3j_2-4j_3}(w)x^n
\ .
\end{align*}
Setting $m=n-2j_1-3j_2-4j_3$ this becomes
\begin{eqnarray*}
&=&\sum_{j_1+j_2+j_3 \ge
0}\frac{(-w)^{j_1+j_2+j_3}}{j_1!j_2!j_3!}x^{2j_1+3j_2+4j_3}\sum_{m
\ge
0}\frac{(m+j_1+j_2+j_3)!}{m!}\varphi_m(w)x^m\\
&=&\sum_{m \ge 0}\left[\sum_{j_1+j_2+j_3\ge 0}{m+j_1+j_2+j_3
\choose
m,j_1,j_2,j_3}(-wx^2)^{j_1}(-wx^3)^{j_2}(-wx^4)^{j_3}\right]
\varphi_{m}(w)x^m\\
&=&\sum_{m \ge 0}\varphi_m(w)x^m\left(\frac{1}{1+wx^2+wx^3+wx^4}
\right)^{m+1}\\
&=&\frac{1}{1+wx^2+wx^3+wx^4}\sum_{m \ge
0}\varphi_m(w)\left(\frac{x}{1+wx^2+wx^3+wx^4}\right)^m \ .
\end{eqnarray*}
Next we compute
\begin{eqnarray*}
\sum_{m \ge 0}\varphi_m(w)y^m & = & \int_{0}^{\infty}\sum_{m
\ge 0}\varphi_m(w)\frac{(xy)^m}{m!}e^{-x}dx\\
&=&\int_{0}^{\infty}\sum_{m\ge 0}\sum_{h\le \frac{m}{2}}{m \choose
2h}f_k(2h,0)w^h\frac{(xy)^m}{m!}e^{-x}dx\\
&=&\int_{0}^{\infty}\sum_{m\ge 0}\sum_{h\le
\frac{m}{2}}f_k(2h,0)w^h\frac{(xy)^{2h}}{(2h)!}
\frac{(xy)^{m-2h}}{(m-2h)!}e^{-x}dx\\
&=&\int_{0}^{\infty}\sum_{h\ge
0}f_k(2h,0)\frac{(\sqrt{w}\,xy)^{2h}}{(2h)!}\sum_{m\ge
2h}\frac{(xy)^{m-2h}}{(m-2h)!}e^{-x}dx\\
&=&\sum_{n \ge
0}f_k(2n,0)\frac{(\sqrt{w}y)^{2n}}{(2n)!}\int_{0}^{\infty}e^{-(1-y)x}
x^{2n}dx\\
&=&\sum_{n \ge
0}f_k(2n,0)\frac{(\sqrt{w}y)^{2n}}{(2n)!}
\frac{\int_{0}^{\infty}e^{-(1-y)x}((1-y)x)^{2n}d((1-y)x)}{(1-y)^{2n+1}}\\
&=&\frac{1}{1-y}\sum_{n \ge
0}f_k(2n,0)\left(\frac{\sqrt{w}y}{1-y}\right)^{2n} \ .
\end{eqnarray*}
Therefore the bivariate generating function can be written as
\begin{eqnarray*}
\sum_{n\ge 0}\sum_{h \le \frac{n}{2}}{\sf T}_k^{*}(n,h)w^hx^n
=\frac{1}{v(x)}\sum_{n \ge
0}f_k(2n,0)\left(\frac{\sqrt{w}\,x}{v(x)}\right)^{2n} \ ,
\end{eqnarray*}
whence Claim $1$. In view of eq.~(\ref{E:c}) and Claim$1$ we
arrive at
\begin{equation}\label{E:h5}
\quad\sum_{n \ge 0}\sum_{0\le h \le \frac{n}{2}}{\sf
C}_{k}^{*}(n,h)\left( \frac{w}{1-wx^2}\right)^hx^n
=\frac{1}{v(x)}\sum_{n \ge
0}f_k(2n,0)\left(\frac{\sqrt{w}\,x}{v(x)}\right)^{2n} \ .
\end{equation}
By definition of $w_0=w_0(x)$ have
\begin{equation}\label{E:w0-x}
\frac{(x^2)^{\sigma-1}}{1-x^2}=\frac{w_0}{1-w_0x^2}.
\end{equation}
According to eq.(\ref{E:d}), eq.(\ref{E:w0-x}) and eq.(\ref{E:h5})
this allows us to derive
\begin{eqnarray*}
{\bf T}^{[4] }_{k,\sigma}(x) & = & \sum_{n \ge
0}\sum_{0\le h \le \frac{n}{2}}{\sf C}^{*}_{k}(n,h)\left(\frac{
(x^2)^{\sigma-1}}{1-x^2}\right)^hx^n \\
&=&\sum_{n\ge 0}\sum_{0\le h\le
\frac{n}{2}}{\sf C}^{*}_{k}(n,h)\left(\frac{w_0}{1-w_0x^2}\right)^hx^n\\
&=&\frac{1}{v_0(x)}\sum_{n \ge
0}f_k(2n,0)\left(\frac{\sqrt{w_0}x}{v_0(x)}\right)^{2n}\ ,
\end{eqnarray*}
whence (\ref{E:functional-arc-3}).
Let ${\bf V}_{k}(x)=\sum_{n \ge 0}f_k(2n,0)\left(\frac{
\sqrt{w_0}\,x}{v_0(x)}\right)^{2n}$.\\
%%%%%%%%%%%%%%%%%%%%%%%%%%%%%%%%%%%%%%%%%%%%%%%%%%%%%%%%%%%%%%%%%%%%%%%
{\it Claim $2$.} The unique, minimal, positive, real solution of
\begin{equation}\label{E:root}
\vartheta_{\sigma}(x)=\frac{\sqrt{w_0}x}{v_0(x)} = \rho_k\ ,\quad
\text{\it for}\quad k=3,4,\dots,9
\end{equation}
denoted by $\gamma_{k,\sigma}^{[4]}$ is the unique dominant singularity of
${\bf T}_{k,\sigma}^{[4]}(x)$. \\
%%%%%%%%%%%%%%%%%%%%%%%%%%%%%%%%%%%%%%%%%%%%%%%%%%%%%%%%%%%%%%%%%%%%%%%
Clearly, a dominant singularity of $\frac{1}{v_0(x)} {\bf V}_k(x)$
is either a singularity of ${\bf V}_k(x)$ or $\frac{1}{v_0(x)}$.
Suppose there exists some singularity $\zeta\in\mathbb{C}$ which
is a pole of $\frac{1}{v_0(x)}$. By construction $\zeta\neq 0$ and
$\zeta$ is necessarily a non-finite singularity of ${\bf V}_k(x)$.
If $\vert \zeta\vert \le \gamma_{k,\sigma}^{[4]}$, then we arrive
at the contradiction
$$
\vert {\bf V}_k(\zeta)\vert>\vert {\bf V}_k(\gamma_{k,\sigma}^{[4]})\vert
\ge  {\bf V}_k(\vert \zeta\vert)
$$
since ${\bf V}_k(\zeta)$ is
not finite and ${\bf V}_k(\gamma_{k,\sigma}^{[4]})=\sum_{n\ge
0}f_k(2n,0)\rho_k^{2n}<\infty$. Therefore all dominant
singularities of ${\bf T}_{k,\sigma}^{[4]}(x)$ are singularities
of ${\bf V}_k(x)$. According to Pringsheim's
Theorem~\cite{Titmarsh:39}, ${\bf T}_{k,\sigma}^{[4]}(x)$ has a
dominant positive real singularity which by construction equals
$\gamma_{k,\sigma}^{[4]}$ being the minimal positive real solution
of eq.~(\ref{E:root}). To prove this, we use that for $3\le k\le
9$, the generating function ${\bf F}_k(x)$ has only the two
dominant singularities $\pm \rho_k$, see Section~\ref{S:pre},
Tab.~\ref{T:111}. Furthermore we verify that for $3\le k\le 9$,
$\gamma_{k,\sigma}^{[4]}$, has strictly smaller modulus than all
solutions of $\vartheta_{\sigma}(z)=-\rho_k$, whence Claim $2$.
Accordingly, Theorem~\ref{T:realdeal} applies and we have
\begin{equation}
{\sf T}_{k,\sigma}^{[4]}(n) \sim c_k n^{-(k-1)^2-\frac{k-1}{2}}
\left(\frac{1}{\gamma_{k,\sigma}^{[4]}}\right)^{n} \quad \text{\rm
for some constant $c_k$}
\end{equation}
completing the proof of Theorem~\ref{T:arc-3}.
\end{proof}

%%%
%%%%%%%%%%%%%%%%%%%%%%%%%%%%%%%%%%%%%%%%%%%%%%%%%%%%%%%%%%%%%%%%%%%%%%%%%%%
%%%

{\bf Acknowledgments.}
%%%
%%%%%%%%%%%%%%%%%%%%%%%%%%%%%%%%%%%%%%%%%%%%%%%%%%%%%%%%%%%%%%%%%%%%%%%%%%
%%%
We are grateful to Hillary S. W. Han, Fenix. W. D. Huang, Emma Y.
Jin and Linda Y. M. Li for their help. This work was supported by
the 973 Project, the PCSIRT Project of the Ministry of Education,
the Ministry of Science and Technology, and the National Science
Foundation of China.

\bibliographystyle{plain}

%%%
%%%%%%%%%%%%%%%%%%%%%%%%%%%%%%%%%%%%%%%%%%%%%%%%%%%%%%%%%%%%%%%%%%%%%%%%%%
%%%

%%%
%%%%%%%%%%%%%%%%%%%%%%%%%%%%%%%%%%%%%%%%%%%%%%%%%%%%%%%%%%%%%%%%%%%%%%%%%%
%%%

\end{document}